\documentclass[twoside,12pt]{amsart}

\usepackage{amsmath}
\usepackage{amsthm}
\usepackage{amsfonts}
\usepackage{amssymb}
\usepackage{latexsym}
\usepackage{float}
\usepackage{upref}
\restylefloat{figure}
\usepackage{mathptmx}
\usepackage{graphicx}
\usepackage{color}

\makeatletter
\def\serieslogo@{}
\makeatother \makeatletter
\def\@setcopyright{}

\def\grad{\nabla}

\def\poly{\mathcal{P}} 

\def\v{{\mathbf v}}
\def\vo{{{\mathbf v}_0}}
\def\x{{\mathbf x}}
\def\xo{{{\mathbf x}_0}}
\def\y{{\mathbf y}}

\def\p{p}

\newtheorem{thm}{Theorem}[section]

\newtheorem{prop}[thm]{Proposition}

\theoremstyle{assumption}

\theoremstyle{definition}
\newtheorem{defn}{Definition}[section]
\newtheorem{rem}{Remark}[section]

\numberwithin{equation}{section}
\numberwithin{thm}{section}
\numberwithin{rem}{section}
\numberwithin{defn}{section}
\numberwithin{assumption}{section}

\begin{document}
\footnotetext{{\it Date:} \today.

{\it 2000 Mathematics Subject Classification:} 65M60, 92C45.

{\it Keywords:} Flocking, kinetic equations, clusters,
$\delta$-singularity, discontinuous Galerkin method, positivity
preserving.

{\bf Acknowledgment.} This work is supported by NSF grants RNMS11-07444
(KI-Net).}

\title[A DG method on kinetic flocking models]
{A discontinuous Galerkin method on kinetic flocking models} 

\author{Changhui Tan}
\address{Changhui Tan\\
Center of Scientific Computation and Mathematical Modeling (CSCAMM)\\
University of Maryland \\
College Park, MD, 20742-4015\\
USA}
\email{ctan@cscamm.umd.edu}

\setcounter{page}{1}

\begin{abstract}
We study kinetic representations of flocking models. They arise from
agent-based models for self-organized dynamics, such as
Cucker-Smale \cite{CuS} and Motsch-Tadmor \cite{MT} models. 
We prove flocking behavior for the kinetic descriptions of flocking
systems, which indicates a concentration in velocity variable in infinite
time. We propose a discontinuous Galerkin method to treat the
asymptotic $\delta$-singularity, and construct
high order positive preserving scheme to solve kinetic flocking systems. 
\end{abstract}

\maketitle
\vspace*{-0.8cm}
\tableofcontents

\section{Introduction}

We are concern with the following Vlasov-type kinetic equation
\begin{subequations}\label{eq:main}
\begin{equation}
\partial_tf+\v\cdot\grad_\x f+\grad_\v\cdot Q(f,f)=0,
\end{equation}
where $f=f(t,\x,\v)$ represents number density, and
the binary interaction $Q(f,f)$ is non-local in space, which is
expressed in the form
\begin{equation}
Q(f,f)=fL[f],\quad L[f](t,\x,\v)=\iint
\frac{\phi(|\x-\y|)}{\Phi(t,\x)} (\v^*-\v)f(t,\y,\v^*)d\y d\v^*.
\end{equation} 
\end{subequations}

This system arises as a mean-field kinetic description of agent-based
self-organized dynamics
\[\displaystyle\dot{\x}_i=\v_i,\quad
\dot{\v}_i=\sum_{j=1}^N a(\x_i,\x_j)(\v_j-\v_i).\]
It discribes the behavior that agents align with their
neighbors and self-organize to finite many clusters,
through an interaction law characterized by a kernel
$a(\cdot,\cdot)$. In particular, it reveals the novel \emph{flocking
  phenomenon} where agents, \textit{e.g.} birds, fishes, 
 organize into an ordered motion and flock into one cluster.

A pioneering work on flocking dynamics is due to Cucker and
Smale (CS)
in \cite{CuS}, where they propose a symmetric interaction kernel
\[a(\x_i,\x_j)=\frac{\phi(|\x_i-\x_j|)}{N}.\]
Here, $\phi$ is called \emph{influence function}, which characterizes
the influences between two agents. It is natural to assume that the
strength of interaction is determined by the physical distance
between agents: larger distance implies weaker influence.
Hence, we assume that
$\phi=\phi(r)$ is a bounded decreasing function on $[0,\infty)$.
Without loss of generality, we set $\phi(0)=1$ throughout the paper.

In particular, if $\phi$
decreases slow enough at infinity, namely
\begin{equation}\label{eq:phi}
\int^\infty\phi(r)dr=\infty,
\end{equation}
CS system enjoys unconditionally flocking property: all agents
tend to have the same asymptotic velocity, regardless of initial
configurations, consult \textit{e.g.} \cite{HL}. 

Another celebrated model is proposed by Motsch and Tadmor (MT)
in \cite{MT}, with interaction kernel
\[a(\x_i,\x_j)=\frac{\phi(|\x_i-\x_j|)}{\Phi_i},\quad
\Phi_i=\sum_{j=1}^N\phi(|\x_i-\x_j|).\]
With the new normalization by $\Phi_i$ as opposed to $N$ in CS system,
MT model performs better in the far-from-equilibrium scenario, consult
\cite{MT} and section \ref{sec:vs} below. Despite the fact that the interaction
kernel is asymmetric, and momentum is not conserved, it is proved in
\cite{MT} that MT system has
unconditional flocking property, under the same assumption
\eqref{eq:phi} on the influence function.

When number of agents $N$ becomes large, it is more convenient to
study the associated kinetic mean-field representation \eqref{eq:main}, which is
formally derived in \cite{HT,MT}. For CS and MT models, the
normalization factor $\Phi$ takes the form
\[\Phi(t,\x)\begin{cases}
\equiv m&\text{for CS model}\\
=\displaystyle\iint\phi(|\x-\y|)f(t,\y,\v)d\y d\v&\text{for MT model}
\end{cases},\]
where $m:=\displaystyle\iint f(t,\y,\v)d\y d\v$ is the total mass
which is conserved in time.

\medskip

The goal of this paper is to study these two flocking models in
kinetic level. Our first result, stated in theorem \ref{thm:flocking}, shows global
existence of classical solution to the main system \eqref{eq:main}, as well as
the long time behavior of the solution:
unconditional flocking under assumption \eqref{eq:phi}. For CS system, the result is 
well-established in \cite{CFRT,HT}. We give an alternative proof
for both CS and MT systems, employing the idea of \cite{MT} in
analogy with the agent-based models. Similar argument can be made for
hydrodynamic flocking models as well, consult \cite{TT}.

Our second main result concerns with numerical implementation of
system \eqref{eq:main}. Despite being smooth for all finite time, the
asymptotic behavior of the solution is the formation of clusters, and
in particular, flocking, under assumption \eqref{eq:phi}.
This implies concentrations in $\v$ as time approaches
infinity. Such $\delta$-singularity is addressed in many systems, from
finite-time concentration in aggregation systems {\it e.g.} \cite{BCL}, to formation of
$\delta$-shocks in Euler equations {\it e.g.} \cite{CL}.

In particular, there are lots of development on numerical
implementation of kinetic systems with singularities of different
types. We refer readers to a recent review
\cite{DP} and references therein. 
Many techniques use smooth approximations for the singularity. 
They suffer large errors as the solution becomes more and more singular. 
For instance, spectral
method is widely used to solve kinetic systems. It is very
accurate and efficient (especially for our system as it has a
convolution structure). However, when solution becomes singular, the
method is unstable, due to Gibbs phenomenon.

We design a discontinuous Galerkin (DG) method to solve the flocking systems
numerically. Discontinuous Galerkin methods are first introduced by
Reed and Hill in \cite{RH} and has many succesful applications in hyperbolic
conservation laws. The idea is to use piecewise
polynomials to approximate the solution in the weak sense. The use of
weak formulation of the solution overcomes the inaccuracy of the
scheme. Moreover, we prove in theorem \ref{thm:pp} that our scheme is
stable, under an appropriate limiter \cite{SZ}.
The efficiency of DG method on $\delta$-singlarity has been
studied in \cite{YS} and more applications are discussed in
\cite{YWS}.
\medskip

The rest of the paper is organized as follows. We first prove flocking
properties for the main system \eqref{eq:main} in section
\ref{sec:flock}. The numerical implementation for the system is developed
in section \ref{sec:DG}. We design DG schemes of second, third or
higher in $v$, and prove $L^1$-stability of the schemes.
Some examples are provided in section
\ref{sec:example} to demonstrate the good performance of our high
order DG schemes, for capturing flocking, as well as clustering
phenomena. In particular, we compare CS and MT setups under a
far-from-equilibrium initial configuration. As addressed in \cite{MT},
MT model has a better performance, in the sense of converging to the
expected flock.

\section{Kinetic description of flocking models}\label{sec:flock}
To illustrate flocking in kinetic level, we first define the total
variation in position $x$ and velocity $v$:
\[S(t):=\sup_{(x,v),(y,v^*)\in\text{supp}f(t)}|x-y|,\quad
V(t):=\sup_{(x,v),(y,v^*)\in\text{supp}f(t)}|v-v^*|.\]

Flocking can be represented using the following definition. There are two
key aspects included: agents tend to have the same velocity as others,
they won't go apart in large time.
\begin{defn}[Kinetic flocking]
We say a solution $f(t,x,v)$ converges to a flock in the kinetic
level, if $S(t)$ remains bounded in all time, and $V(t)$ decays to
0 asymptotically, namely,
\[S(t)\leq D, ~~\forall t\geq0;\quad V(t)\to0 \text{ as } t\to\infty.\]
\end{defn}

We prove the flocking property of both Cucker-Smale and Motsch-Tadmor
model in the kinetic level. 
\begin{thm}[Unconditional flocking]\label{thm:flocking}
Consider kinetic flocking system \eqref{eq:main} with CS or
MT setup. Suppose the influence function $\phi$ satisfies 
\eqref{eq:phi}. Then, for any initial profile $f_0\in C^1\cap W^{1,\infty}$, there exists a
unique strong solution of the system in all time and the solution
converges to a flock.
\end{thm}

First, we claim that all $C^1$ solutions converges to a flock. With
the assumption of smoothness, we are able to study the characteristic
paths. The following decay estimates play an important rule toward the
proof of flocking.

\begin{prop}[Decay estimates of flocks] Suppose $f$ is the strong solution of
  the system \eqref{eq:main}. Then, 
\begin{subequations}\label{eq:flock}
\begin{align}
\frac{d}{dt} S(t)&\leq V(t),\label{eq:dX}\\
\frac{d}{dt} V(t)&\leq -\phi(S(t)) V(t).\label{eq:dV}
\end{align}
\end{subequations}
\end{prop}
\begin{proof}
The characteristic curve of the system reads $(\x(t),\v(t))$ where
\[\frac{d}{dt}\x(t)=\v(t),\quad\frac{d}{dt}v(t)=L[f](t,\x(t),\v(t)).\]
We consider two characteristics $(\x(t),\v(t))$ and
$(\y(t),\v^*(t))$, both starting inside the support of $f_0$.
To simplify the notations, we omit the time variable throughout the
proof, unless necessary.

\medskip{\bf Step 1: proof of \eqref{eq:dX}.} Compute
\[\frac{d}{dt}|\x-\y|^2=2\langle \x-\y,\v-\v^*\rangle\leq 2SV.\]

By taking the supreme of the left hand side among all $\x, \y$, the
inequality yields \eqref{eq:dX}.

\medskip{\bf Step 2: proof of \eqref{eq:dV}.} Similar with step 1, compute
\[\frac{d}{dt}|\v-\v^*|^2=2\big\langle \v-\v^*,L[f](\x,\v)-L[f](\y,\v^*)\big\rangle.\]

We claim the following key estimate
\begin{equation}\label{eq:flockkey}
L[f](\x,\v)-L[f](\y,\v^*)\leq (1-\phi(S))V-(\v-\v^*),
\end{equation}
for all $(\x,\v), (\y,\v^*)$ in the support of $f$. It yields
\[\frac{d}{dt}|\v-\v^*|^2\leq 2(1-\phi(S))|\v-\v^*|V-2|\v-\v^*|^2.\]
Take $\v,\v^*$ where $|\v-\v^*|\to V$, we end up with \eqref{eq:dV}.

\medskip{\bf Step 3: proof of the key estimate \eqref{eq:flockkey}.}
Given any pairs $(\x,\v)$ and $(\y,\v^*)$ inside the support of $f$, define 
\[b(t,\x,\v,\y,\v^*):=\frac{\phi(|\x-\y|)f(t,\y,\v^*)}{\Phi(t,\x)}+
\left(1-\iint\frac{\phi(|\x-\y|)}{\Phi(t,\x)}f(t,\y,\v)d\y
  d\v\right)\delta_0(\x-\y)\delta_0(\v-\v^*),\]
where $\delta_0$ is the Dirac delta at the origin. Such function $b$
enjoys the following properties
\begin{itemize}
\item[(P1)] $\displaystyle\iint b(t,\x,\v,\y,\v^*)d\y d\v^*=1$, for all $t$,
\item[(P2)] $\displaystyle\iint b(t,\x,\v,\y,\v^*)(\v^*-\v)d\y d\v^*=L[f](t,\x,\v)$,
\item[(P3)] There exists a function $\eta(t,\v^*)$ such that 
  \begin{itemize}
    \item $\displaystyle\int b(t,\x,\v,\y,\v^*)dy\geq \eta(t,\v^*)$ for all $t, \x$ and $\v$,
    \item $\displaystyle\int\eta(t,\v^*)d\v^*=\phi(S(t))>0$, for all $t$.
  \end{itemize}
\end{itemize}

It is worth noting that the second term in $b$ is 0 under MT setup.
For CS setup, it is a positive delta measure. The sole purpose of
adding this term is to satisfy (P1). Hence, the main ingredient of $b$
is the first term.

The first two properties (P1) and (P2) are easy to check. Details are
left to readers. For (P3), a valid choice of $\eta$ is 
\[\eta(\v^*)=\frac{\phi(S)}{m}\int f(\y,\v^*)d\y.\]
With this $\eta$, we check the first condition
\[\int b(\x,\v,\y,\v^*)d\y=\int\frac{\phi(|\x-\y|)}{\Phi(\x)}
f(\y,\v^*)d\y\geq \frac{\phi(S)}{m}\int f(\y,\v^*)d\y
=\eta(\v^*),\]
thanks to the 
decreasing property of $\phi$ and the universal assumption
of $\phi(0)=1$, which indicates $\Phi(\x)\leq m$ under both setups.
The second condition is straightforward.

We are ready to prove estimate \eqref{eq:flockkey}. Take $(\x_1,\v_1),
(\x_2,\v_2)$ two characteristics inside the support of $f$. Compute
\begin{align*}
&\quad L[f](\x_1,\v_1)-L[f](\x_2,\v_2)\\
&~\stackrel{(P2)}{=}~\iint\left[b(\x_1,\v_1,\y,\v^*)(\v^*-\v_1)-b(\x_2,\v_2,\y,\v^*)
(\v^*-\v_2)\right]d\y d\v^*\\
&~\stackrel{(P1)}{=}~\iint\left(b(\x_1,\v_1,\y,\v^*)-b(\x_2,\v_2,\y,\v^*)\right)\v^*d\y
d\v^*-(\v_1-\v_2)\\
&\ \ = \ \ \left[\int\left(\int b(\x_1,\v_1,\y,\v^*)d\y-\eta(\v^*)\right)d\v^*
-\int\left(\int b(\x_2,\v_2,\y,\v^*)d\y-\eta(\v^*)\right)d\v^*\right]-(\v_1-\v_2)\\
&\ \ = \ \ (1-\phi(S))\left[\int\hat{b}(\x_1,\v_1,\v^*)\v^*d\v^*
-\int\hat{b}(\x_2,\v_2,\v^*)d\v^*\right]-(\v_1-\v_2).
\end{align*}
Here, $\hat{b}$ is defined as $\hat{b}(\x,\v,\v^*)=\displaystyle\frac{\int
b(\x,\v,\y,\v^*)d\y-\eta(\v^*)}{1-\phi(S)}$. From (P1) and (P3), we know $\hat{b}$ is
positive, supported inside the support of $f$ in $\v$, and $\displaystyle\int
\hat{b}(\x,\v,\v^*)d\v^*=1$ for all $(\x,\v)$. Therefore, $\displaystyle
\int\hat{b}(\x,\v,\v^*)\v^*d\v^*$ lies inside the convex envelope of the
support of $f$ in $\v$. Hence, 
\[\left|\int\hat{b}(\x_1,\v_1,\v^*)\v^*d\v^*-\int\hat{b}(\x_2,\v_2,\v^*)d\v^*\right|\leq
V,\]
and \eqref{eq:flockkey} holds.
\end{proof}

With the decay estimates, we are able to show that the solution of
system \eqref{eq:main} converges to a flock under suitable assumptions
on the influence function.

\begin{thm}[Flock with fast alignment]\label{thm:flock}
Let $f$ be the solution of system \eqref{eq:main}, with initial data
$f_0$ compactly supported, i.e
\[S_0<+\infty \quad\text{and}\quad V_0<+\infty.\]
If the influence function $\phi$ decays sufficiently slow 
\begin{equation}\label{eq:phidecay}
\int_{S_0}^\infty\phi(r)dr>V_0,
\end{equation}
then, $f$ converges to a flock with fast alignment, namely,
there exists a finite number $D$, defined as
\begin{equation}\label{eq:D}
D:=\psi^{-1}(V_0+\psi(S_0)),~~\text{where}~\psi(t)=\int_0^t\psi(s)ds,
\end{equation}
such that
\[\sup_{t\geq 0}S(t)\leq D,\quad V(t)\leq V_0 e^{-\phi(D)t}.\]
\end{thm} 
\begin{rem}
1.  The idea of the proof is followed from
\cite{HL}. Consult \cite[Proof of theorem 2.1]{TT} for more details.
 Note that $V(t)$ decays to zero exponentially fast. We call
this fast alignment.

2. Condition \eqref{eq:phidecay} is automatically satisfied if we assume
$\phi$ decays slow at infinity. In fact, with our assumption
\eqref{eq:phi} on $\phi$, \eqref{eq:phidecay} stays true for all
initial configurations with finite $S_0$ and $V_0$. Hence, we prove
\emph{unconditional flocking}.
\end{rem}

Next, we show $f\in C^1$ in all time. For Vlasov-type equations, the
proof is quite standard, see {\it e.g.} \cite{HT} for CS system.
\begin{prop}\label{prop:classical}
Consider \eqref{eq:main} with initial $f_0\in C^1\cap
W^{1,\infty}$. Then, there exists a unique solution $f\in C([0,T],
C^1\cap W^{1,\infty})$, for any time $T$.
\end{prop}
\begin{rem}
Formally, by integrating the velocity variable, we can get
corresponding hydrodynamic systems of flocking, for both CS and MT
systems. The existence of global strong solution is
not as straightforward as the kinetic system, due to the nonlinear conviction
term. We refer to \cite{TT} for existence and flocking properties of
the hydrodynamic flocking systems, where a critical threshold is
introduced to guarantee global strong solutions. 
\end{rem}

\begin{proof}[Proof of proposition \ref{prop:classical}]
Take characteristic path $(\x(t),\v(t))$ starting at $(\xo,\vo)$.
\begin{align*}
\dot{\x}(t,\xo,\vo)&=\v(t,\xo,\vo),\\
\dot{\v}(t,\xo,\vo)&=L(f)(t,\x(t,\xo,\vo),\v(t,\xo,\vo)).
\end{align*}
Define the Jacobian 
\[J(t,\xo,\vo)=\begin{bmatrix}\partial_\xo \x&\partial_\vo \x\\\partial_\xo
  \v&\partial_\vo \v\end{bmatrix}, \quad
A(t,\xo,\vo)=\begin{bmatrix}0&1\\ \partial_\xo L(f)&\partial_\vo L(f)\end{bmatrix}.\]
It is easy to check that
\begin{align*}
\dot{J}(t,\xo,\vo)&=A(t,\x,\v) J(t,\xo,\vo),\quad
J(\xo,\vo,0)\equiv\mathbb{I}_{2n\times2n},\\
\dot{J}^{-1}(t,\xo,\vo)&=- J^{-1}(t,\xo,\vo) A(t,\x,\v),\quad
J^{-1}(\xo,\vo,0)\equiv\mathbb{I}_{2n\times2n},\\
det~J(t,\xo,\vo)&=\exp\left(\int_0^t\text{tr}A(s,\x(s,\xo,\vo), \v(s,\xo,\vo))ds\right).
\end{align*}

Along the characgeristics, we have
\[f(t,\x(t,\xo,\vo),\v(t,\xo,\vo))=f_0(\xo,\vo)(det~J(t,\xo,\vo))^{-1}.\]
It is sufficient to prove that $f(t,\cdot,\cdot)\in L^\infty_{\x,\v}$ in any finite
time as long as $\|A\|_{L^\infty}$ is finite.

To this end, we check for CS,
\begin{align*}
\left|\partial_\x L(f)\right|&
=\left|\frac{1}{m}\iint\partial_\x\phi(|\x-\y|)(\v^*-\v)f(\y,\v^*)d\y d\v^*\right|
\leq\|\phi\|_{\dot{W}^{1,\infty}}V(t) \leq\|\phi\|_{\dot{W}^{1,\infty}}V_0,\\
\left|\partial_\v L(f) \right|&
=\left|-\frac{1}{m}\iint\phi(|\x-\y|)f(\y,\v^*)d\y d\v^*\right|\leq 1.
\end{align*}
For MT,
\begin{align*}
\left|\partial_\x L(f)\right|&
=\left|\iint\partial_\x\left(\frac{\phi(|\x-\y|)}{\Phi(\x)}\right)(\v^*-\v)f(\y,\v^*)d\y
  d\v^*\right|\\
&\leq V(t)\left[
\frac{1}{\Phi(\x)}\iint\partial_\x\phi(|\x-\y|)f(\y,\v^*)d\y d\v^*
+\frac{|\partial_\x\Phi(\x)|}{\Phi(\x)^2}\iint\phi(|\x-\y|)f(\y,\v^*)d\y
dv^*\right]\\
& \leq 2V_0\frac{\left|\partial_\x\Phi(\x)\right|}{\Phi(\x)}
\leq\frac{2m\|\phi\|_{\dot{W}^{1,\infty}}V_0 }{\phi(D)},\\
\partial_\v L(f)&
=-\iint\frac{\phi(|\x-\y|)}{\Phi(\x)}f(\y,\v^*)d\y d\v^*
=-1.
\end{align*}

For classical solutions, we need to bound $\grad_{(\x,\v)} f$. It
fact, we have
\begin{align*}
(\grad f)&(t,\x(t),\v(t))=J^{-1}(t)\grad
f_0(\xo,\vo)\exp\left(-\int_0^t\text{tr}A(s,\x(s),\v(s))ds\right)\\
&+f_0(\xo,\vo)\exp\left(-\int_0^t~trA(s,\x(s),\v(s))ds\right)\int_0^t
J(s)~(\grad\text{tr}A)(s,\x(s),\v(s))ds.
\end{align*}

As $\|A\|_{L^\infty}$ is bounded, it is clear that $J$ and $J^{-1}$
are bounded pointwise by $e^{Ct}$. To obtain boundedness of $\grad f$,
we are left to estimate $\displaystyle\grad\text{tr}A=\grad\partial_\v
L(f)$. Notice that $L(f)$ is linear in $\v$ for both setups. Hence,
$\partial_\v^2L(f)=0$.

Compute $\partial_\x\partial_\v L(f)$ for CS:
\[\left|\partial_\x\partial_\v L(f)\right|
=\left|-\frac{1}{m}\iint\partial_\x\phi(|\x-\y|)f(\y,\v^*)d\y d\v^*\right|
\leq\|\phi\|_{\dot{W}^{1,\infty}}.\]
For MT, as $\partial_\v L(f)=-1$, it directly implies
$\partial_\x\partial_\v L(f)=0$.

We end up with global existence of classical solutions with
\[\|f(t,\cdot,\cdot)\|_{W^{1,\infty}}\leq \|f_0\|_{W^{1,\infty}}e^{Ct}.\]
\end{proof}
Thus, we complete the proof of theorem \ref{thm:flocking}.

\section{A discontinuous Galerkin method}\label{sec:DG}
In this section, we start to discuss the numerical implementation of
kinetic flocking system \eqref{eq:main}. The main goal is to design
high accuracy schemes that are stable as the solution becomes
singular.

As the singularity happens in $\v$ variable due to the alignment
operator, we shall concentrate on the flocking part of the system
\[\partial_t f+\grad_\v\cdot Q(f,f)=0.\] 

We can rewrite the system in the following form
\begin{equation}\label{eq:flocking}
\partial_tf(t,\x,\v)=-\grad_\v\cdot\left(fL[f]\right)=-\grad_\v\cdot\left(f(t,\x,\v)\int
  (\v^*-\v)G(t,\x,\v^*)d\v^*\right),
\end{equation}
where $G$ is defined by
\[G(t,\x,\v)=\frac{1}{\Phi(\x)}\int\phi(|\x-\y|)f(t,\y,\v)d\y.\]
It is easy to check that
\begin{equation}\label{eq:Gprop}
\int G(t,\x,\v)d\v\leq1,
\end{equation}
for all $\x$ and $t$. In particular, the equality holds under MT setup.

As \eqref{eq:flocking} is homogeneous in $\x$, we omit the $\x$ dependency
for simplicity from now on. 

\subsection{The DG framework}
The idea of the discontinuous Galerkin method is to use piecewise
polynomial to approximate the solution. We take 1D as an easy illustration.

We partition the computational domain $\Omega=[a,b]$ on $v$ into $N$ cells
$\{I_j\}_{j=1}^N$ 
\[I_j=\left(v_{j-1/2}, v_{j+1/2}\right),\quad v_j=a+(j-1/2)\Delta v,
\quad\Delta v=\frac{b-a}{N},\]
with uniform mesh size $h:=\Delta v$ for simplicity. The space we are
working with is
\[V_h:=\left\{f~:~\text{For all } j=1,\cdots,N, ~f|_{I_j}\in\poly_k\right\},\]
where $\poly_k$ denotes polynomial of degree at most $k$. The weak
formulation of \eqref{eq:flocking} reads
\begin{equation}\label{eq:weak}
\frac{d}{dt}\int_{I_j} f(v)\p(v) dv = -\p
fL[f]\left|_{v_{j-1/2}}^{v_{j+1/2}}\right.+\int_{I_j}fL[f]\phi'
dv,\quad \forall \p=\p(v)\in V_h.
\end{equation}
The DG scheme is to find $f\in V_h$ which satisfies \eqref{eq:weak}.

If we apply test function $\p(v)=1$ on \eqref{eq:weak}, we
get
\[\frac{d}{dt}\bar{f}_j=-\frac{1}{h}fL[f]\left|_{v_{j-1/2}}^{v_{j+1/2}}\right.,\]
where $\bar{f}_j$ is the cell average of $I_j$.
With a forward Euler scheme in time, this becomes the classical finite
volume method, namely
\[\bar{f}_j(t+\Delta t)=\bar{f}_j(t)+\frac{\Delta t}{h}\left[
f(v_{j-1/2}^+)\cdot L[f](v_{j-1/2})-
f(v_{j+1/2}^-)\cdot L[f](v_{j+1/2})\right].\]

\begin{subequations}\label{eq:fv}
The heart of the matter is to approximate the flux at the cell interfaces.
To ensure the conservation law, we modify the scheme using a numerical
flux
\begin{equation}\label{eq:test0}
\bar{f}_j(t+\Delta t)=\bar{f}_j(t)+\frac{\Delta t}{h}\left[
\hat{f}(v_{j-1/2})\cdot L[f](v_{j-1/2})
-\hat{f}(v_{j+1/2})\cdot L[f](v_{j+1/2})\right]
\end{equation}
so that the outflux and influx at the same interface add up to
zero. Note that $L$ is a global operator on $f$, and $L[f]$ is
continuous at the interface, we need to compute $L[f]$ using
information from all cells. Then, with fixed $L[f](v_{j+1/2})$, the
flux is linear in $f$. We use upwind fluxes where
\begin{equation}
\hat{f}_{j+1/2}:=\hat{f}(v_{j+1/2})=\begin{cases}
f(v_{j+1/2}^-)&\text{if }L[f](v_{j+1/2})\geq0\\f(v_{j+1/2}^+)&\text{if
}L[f](v_{j+1/2})<0\end{cases}.
\end{equation}
\end{subequations}
\begin{rem}
We use monotone numerical flux for DG scheme. In our simple
case when the flux is linear, some widely used flux such as Godunov
flux, Lax-Friedrich flux coincide with the upwind flux.
\end{rem}

\subsection{A first order scheme}\label{sec:first}
Let us consider the simple case when $k=0$. A piecewise constant
approximation yields first order accuracy. To obtain
$\bar{f_j}(t+\Delta t)$, we apply scheme \eqref{eq:fv} with
\[f(v_{j+1/2}^+)=\bar{f}_{j+1},\quad f(v_{j+1/2}^-)=\bar{f}_j,\]
 as $v$ is a constant in each cell. We are left with computing
$L[f]$. As $f$ is piecewise constant in $v$ for all $x$, clearly $G$
is also a piecewise constant in $v$. Hence,
\begin{align*}
L[f](v_{j+1/2})=\int_\Omega(v^*-v_{j+1/2})G(v^*)dv^*=
\sum_{l=1}^N\bar{G}_l\int_{I_l}(v^*-v_{j+1/2})dv^*=h^2\sum_{l=1}^N
(l-j-1/2)\bar{G}_l,
\end{align*}
where $\bar{G}_l$ is the value of $G$ in $I_l$. We can use any first
order numerical integration on $x$ to compute $\bar{G}_l$ from
$\bar{f}_l$.

We prove the positivity preserving property of the first order
scheme, which ensures $L^1$ stability of the numerical solution.
\begin{prop}\label{prop:CFL}
Suppose $\bar{f}_j(t)>0$ for all $j$. Applying the first order scheme,
we have $\bar{f}_j(t+\Delta t)>0$ under CFL condition
\begin{equation}\label{eq:CFL}
\frac{\Delta t}{h}\max_j\left| L[f](v_{j+1/2})\right|<\frac{1}{2}.
\end{equation}
\end{prop}
\begin{proof}
Rewrite \eqref{eq:test0} as following
\[\bar{f}_j(t+\Delta t)=\frac{1}{2}\left[\bar{f}_j(t)+\frac{2\Delta t}{h}
\hat{f}(v_{j-1/2})\cdot L[f](v_{j-1/2})\right]+\frac{1}{2}\left[\bar{f}_j(t)
-\frac{2\Delta t}{h}\hat{f}(v_{j+1/2})\cdot L[f](v_{j+1/2})\right].
\]
We will show that both terms are positive under CFL condition.

For the first term, if $L[f](v_{j-1/2})\geq0$, clearly
\[\bar{f}_j(t)+\frac{2\Delta t}{h}\hat{f}(v_{j-1/2})\cdot
L[f](v_{j-1/2})=\bar{f}_j(t)+\frac{2\Delta t}{h}
\bar{f}_{j-1}(t)\cdot L[f](v_{j-1/2})>0.\]
if $L[f](v_{j-1/2})<0$, then under CFL condition, we have
\[\bar{f}_j(t)+\frac{2\Delta t}{h}\hat{f}(v_{j-1/2})\cdot
L[f](v_{j-1/2})=\left[1-\frac{2\Delta
    t}{h}\left|L[f](v_{j-1/2})\right|\right]\bar{f}_j(t)>0.\]

Similarly, the second term is positive under the same CFL
condition. Therefore, $\bar{f}_j(t+\Delta t)>0$, for all $j$.
\end{proof}

\begin{rem}\label{rem:CFL}
The CFL condition \eqref{eq:CFL} depends on time $t$. We can derive a
sufficient CFL condition where the choice of $\Delta t$ is independent
of $t$.

As $G$ is piecewise linear, we deduce from \eqref{eq:Gprop} that
\[\sum_{l=1}^N\bar{G}_l=\int_\Omega G(v)dv\leq1.\]
Hence,
\[\left|L[f](v_{j+1/2})\right|=h^2\left|\sum_{l=1}^N(l-j-1/2)\bar{G}_l\right|\leq
h^2(N-1/2)\sum_{l=1}^N\bar{G}_l=(N-1/2)h< b-a,
\]
for any $j=0,\cdots,N-1$. This implies a sufficient CFL condition
\[\frac{\Delta t}{h}<\frac{1}{2(b-a)}.\] 
\end{rem}

We complete an algorithm solving \eqref{eq:flocking} with first order
accuracy.

\subsection{Higher order DG schemes}
In order to obtain high order accuracy, we apply \eqref{eq:weak}
with test functions with high orders. Choose Legendre polynomials on $I_j$
\[\p_j^{(0)}(v)=1,\quad\p_j^{(1)}(v)=v-v_j,\quad
\p_j^{(2)}(v)=(v-v_j)^2-\frac{1}{12}h^2,\quad\cdots.\] 
Denote
$\displaystyle
f_j^{(l)}=\frac{1}{h^{l+1}}\int_{I_j}f(v)\p_j^{(l)}dv$. Clearly, 
all $f\in\poly_k$ can be determined by $f_j^{(l)}$ for $j=1,\cdots,N$, $l=0,\cdots,k$.
As a matter of fact, we can write $f(v)=\sum_{l=0}^k
a_lf_j^{(l)}\p_j^{(l)}(v)$ for $v\in I_j$, with $a_0=1,
a_1=12/h,a_2=180/h^2$, etc. (Consulting \cite{CS}.)

From \eqref{eq:weak}, we obtain the evolution of $f_j^{(l)}$.
\begin{equation}\label{eq:evolution}
\begin{aligned}
\frac{d}{dt}f_j^{(0)}&=\frac{1}{h}(\hat{f}_{j-1/2}L_{j-1/2}-\hat{f}_{j+1/2}L_{j+1/2}),\\
\frac{d}{dt}f_j^{(1)}&=-\frac{1}{2h}(\hat{f}_{j-1/2}L_{j-1/2}+\hat{f}_{j+1/2}L_{j+1/2})
+\frac{1}{h^2}\int_{I_j}fL[f]dv,\\
\frac{d}{dt}f_j^{(2)}&=\frac{1}{6h}(\hat{f}_{j-1/2}L_{j-1/2}-\hat{f}_{j+1/2}L_{j+1/2})
+\frac{2}{h^3}\int_{I_j}fL[f](v-v_j)dv,
\end{aligned}
\end{equation}
etc. Here, we denote $L_{j\pm 1/2}=L[f](v_{j\pm 1/2})$ for simplicity.

Next, we compute $L_{j+1/2}$ and the two integrals in the dynamics
above, given $f\in V_h$.

For $k=0$, $L_{j+1/2}$ is given in section
\ref{sec:first}. $f_j^{(0)}$ coincide with $\bar{f}_j$.

For $k\geq1$, we use $L^2$-orthogonality property of Legendre
polynomial to compute  
\begin{align*}
L[f](v)~=~&\int (v^*-v)G(v^*)dv^*\\
~=~&\sum_{l=1}^N\int_{I_l}\left[(v_l-v)\p_l^{(0)}(v^*)+\p_l^{(1)}(v^*)\right]\cdot
\left[G_l^{(0)}\p_l^{(0)}(v^*)+\frac{12}{h}G_l^{(1)}\p_l^{(1)}(v^*)+\cdots\right]dv^*\\
~=~&h\sum_{l=1}^N(v_l-v)G_l^{(0)}+h^2\sum_{l=1}^NG_l^{(1)}.
\end{align*}
All other terms of $G(v^*)$ is $L^2$-orthogonal to $v^*-v$ and
have no contribution to $L[f](v)$. This implies
\[L_{j+1/2}~=~h\sum_{l=1}^N(v_l-v_{j+1/2})G_l^{(0)}+h^2\sum_{l=1}^NG_l^{(1)} =
h^2\sum_{l=1}^N \left[(l-j-1/2)G_l^{(0)}+G_l^{(1)}\right].\]

Moreover, $L[f](v)$ is linear in terms of $v$. Again, by
orthogonality, we get
\begin{align*}
\frac{1}{h^2}\int_{I_j}fL[f]dv~=~&\frac{1}{h^2}\int_{I_j}
f(v)\left[\left(h\sum_{l=1}^N(v_l-v_j)G_l^{(0)}+h^2\sum_{l=1}^NG_l^{(1)}\right)\p_j^{(0)}(v)
+\left(-h\sum_{l=1}^NG_l^{(0)}\right)\p_j^{(1)}(v)\right]dv\\
~=~&
h\left\{f_j^{(0)}\sum_{l=1}^N[(l-j)G_l^{(0)}+G_l^{(1)}]
-f_j^{(1)}\sum_{l=1}^NG_l^{(0)}\right\}.
\end{align*}

Finally, for $k\geq2$, 
\[\frac{2}{h^3}\int_{I_j}fL[f](v-v_j)dv=2h\left\{
f_j^{(1)}\sum_{l=1}^N[(l-j)G_l^{(0)}+G_l^{(1)}]
-\left(\frac{1}{12}f_j^{(0)}+f_j^{(2)}\right)\sum_{l=1}^NG_l^{(0)}\right\}.\]

\begin{rem}
As shown above, to compute the right hand side of
\eqref{eq:evolution}, we need to calculate the following sums:
\[\sum_{l=1}^NG_l^{(0)},\quad
\sum_{l=1}^NG_l^{(1)}\quad\text{and}\quad
\sum_{l=1}^N(l-j)G_l^{(0)}.\]
The first two sums are independent of $j$. The third sum has a
convolution structure. Fast convolution solvers could be used to
compute the sum.
\end{rem}

\subsection{Positivity preserving}\label{sec:pos}
One major difficulty of high order schemes is that the reconstructed
solution is not necessarily positive. A negative computational
solution will quickly become unstable. Suitable limiters are needed to
preserve positivity of the numerical solution. We
proceed with the limiter introduced in \cite{SZ}.

First, we extend proposition \ref{prop:CFL} to high order schemes and
prove positivity for $\bar{f}_j$. To proceed, we use Gauss-Lobatto
quadrature points on $I_j$, denoting $\{v_j^i\}_{i=1}^n$. In
particular, $v_j^1=v_{j-1/2}$ and $v_j^n=v_{j+1/2}$. For $f_j$ a
polynomial of degree up to $2n-3$, 
\[\bar{f}_j=\frac{1}{h}\int_{I_j}f_j(v)dv=\frac{1}{h}\sum_{i=1}^n\alpha_if_j(v_j^i),\]
where $\alpha_i$ are Gauss-Lobatto weights. For example, when $n=2$,
$\alpha_1=\alpha_2=1/2$; when $n=3$, $\alpha_1=\alpha_3=1/6$ and
$\alpha_2=2/3$. Note that $\alpha_i$'s are all positive, summing up to
1, and symmetric $\alpha_i=\alpha_{n+1-i}$.

\begin{prop}\label{prop:CFLhigh}
Suppose $f_j(t,v_j^i)>0$ for all Gauss-Lobatto quadrature points
$v_j^i$. Then, for any scheme with forward Euler in time and DG in space  with
order $k\leq 2n-3$,
we have $\bar{f}_j(t+\Delta t)>0$, under CFL condition
\begin{equation}\label{eq:CFLhigh}
\frac{\Delta t}{h}\max_j\left| L_{j+1/2}\right|<\alpha_1.
\end{equation}
In particular, for $k=0,1$, $\alpha_1=1/2$. For $k=2$, $\alpha_1=1/6$.
\end{prop}
\begin{proof}
The dynamic of $\bar{f}_j=f_j^{(0)}$ reads
\begin{align*}
\bar{f}_{j}(t+\Delta t) = & \bar{f}_j(t)+\frac{\Delta
  t}{h}\left[\hat{f}(v_{j-1/2})\cdot L_{j-1/2}-\hat{f}(v_{j+1/2})\cdot
  L_{j+1/2}\right]\\
=& \frac{1}{h}\sum_{i=2}^{n-1}\alpha_if_j(v_j^i)+\alpha_1\left(
f_j(v_{j-1/2})+\frac{\Delta t}{\alpha_1h}\hat{f}(v_{j-1/2})\cdot L_{j-1/2}\right)\\
&+\alpha_n\left(
f_j(v_{j+1/2})-\frac{\Delta t}{\alpha_nh}\hat{f}(v_{j+1/2})\cdot L_{j+1/2}\right).
\end{align*}

We check positivity for the last two terms. For the second term, if
$L_{j-1/2}\geq 0$, clearly
\[f_j(v_{j-1/2})+\frac{\Delta t}{\alpha_1h}\hat{f}(v_{j-1/2})\cdot
L_{j-1/2}=f_j(v_{j-1/2})+\frac{\Delta t}{\alpha_1h}f_{j-1}(v_{j-1/2})\cdot L_{j-1/2}>0.\]
If $L_{j-1/2}<0$, then under CFL condition, we have
\[f_j(v_{j-1/2})+\frac{\Delta t}{\alpha_1h}\hat{f}(v_{j-1/2})\cdot
L_{j-1/2}=
\left[1-\frac{\Delta
    t}{\alpha_1h}|L_{j-1/2}|\right]f_j(v_{j-1/2})>0.\]

Similarly, the third term is positive under the same CFL condition, as
$\alpha_n=\alpha_1$. Therefore, $\bar{f}_j(t+\Delta t)>0$, for all $j$.
\end{proof}

Similar to remark \ref{rem:CFL}, there is a sufficient CFL condition
independent of $t$ for the high order DG scheme. We estimate the
additional part of $L_{j+1/2}$ as below.
\[\left|h^2\sum_{l=1}^NG_l^{(1)}\right|=\frac{1}{\Phi(x)}\left|\sum_{l=1}^N\int\int_{I_l}\phi(|x-y|)f(y,v)(v-v_l)dvdy\right|
\leq \frac{h}{2}\int_\Omega G(x,v)dv\leq\frac{h}{2}.\]
Together with the estimate for the first part (shown in remark
\ref{rem:CFL}), we get
\[\left|L_{j+1/2}\right|\leq \left(N-\frac{1}{2}\right)h+\frac{h}{2}=Nh=(b-a).\]
With the correction term, we have the same bound on $L_{j+1/2}$. It
yields the following sufficient CFL condition
\begin{equation}\label{eq:sCFL}
\frac{\Delta t}{h}<\frac{\alpha_1}{b-a}.
\end{equation}

To make sure $f_j$ is positive at Gauss-Lobatto quadrature points, we
modify $f(t)$ using an interpolation
between the current $f$ and the positive constant $\bar{f}=f^{(0)}$,
namely, in $I_j$ at time $t+\Delta t$,
\[\tilde{f}_j(v)=\theta_j f_j(v)+(1-\theta_j)\bar{f}_j,\]
where $\theta_j\in[0,1]$ to be chosen. When $\theta_j=1$, there is no
modification and high accuracy is preserved. When, $\theta_j=0$, the
modified solution coincides with the first order scheme. Hence, for
higher accuracy, $\theta_j$ should be as large as possible. On the
other hand, we need positivity of $\tilde{f}_j(v_j^i)$, i.e.
\[(\bar{f}_j-f_j(v_j^i))\theta_j <\bar{f}_j,\]
for all $i$. Therefore, we shall choose $\theta_j$ as follows
\[\theta_j=\begin{cases}\displaystyle\frac{\bar{f}_j-\epsilon}{\bar{f}_j-m_j}&\text{if
  } m_j<\epsilon\\~&~\\1&\text{if }m_j\geq\epsilon\end{cases},
\quad\text{where }
m_j:=\min_i f_j(v_j^i),~~ \epsilon=\min\{10^{-13}, \bar{f}_j\}.\]

The modified solution $\tilde{f}_j$ preserves the total mass as
well. It implies $L^1$ stability of the scheme.

We can write the modification in terms of $f_j^{(l)}$ where
\begin{equation}\label{eq:limiter}
\tilde{f}_j^{(0)}=f_j^{(0)},\quad
\tilde{f}_j^{(l)}=\theta_jf_j^{(l)}, l\geq1.
\end{equation} 
Indeed, the modification weakens the
high order correction at several cells to enforce positivity.
But it has been discussed in \cite{SZ} that the order of accuracy
is not strongly affected by this limiter.

We conclude this part with a summary of the stability result for our
high order DG schemes.
\begin{thm}[Positivity preserving]\label{thm:pp}
Consider \eqref{eq:flocking} with initial density $f_0\geq0$.
Then, the solution generated by the DG scheme \eqref{eq:evolution}
with limiter \eqref{eq:limiter} is positive in all 
time, under CFL condition \eqref{eq:sCFL}.
\end{thm}
\begin{rem}
The whole procedure can be extended to multi-dimensional systems. See
{\it e.g.} \cite{SZ2} for examples on this positivity preserving
limiter in multi dimension.
\end{rem}

\subsection{High order time discretization}
In this subsection, we discuss time discretization for the ODE systems
with respect to $f_j^{(l)}$. We already show positivity preserving and
$L^1$ stability for forward Euler time discretization, under CFL
condition \eqref{eq:CFLhigh}. To get high order accuracy in time, we use
strong stability preserving (SSP) Runge-Kutta method \cite{GST}. For
instance, a second order SSP scheme reads
\begin{align*}
f_{[1]}&~=~\mathtt{FE}(f(t),\Delta t)\\
f(t+\Delta t)&~=~ \frac{1}{2} f(t)+\frac{1}{2}
\mathtt{FE}(f_{[1]}, \Delta t),
\end{align*}
and a third order SSP scheme reads
\begin{align*}
f_{[1]}&~=~\mathtt{FE}(f(t),\Delta t)\\
f_{[2]}&~=~\frac{3}{4}f(t)+\frac{1}{4}\mathtt{FE}(f_{[1]},\Delta t)\\
f(t+\Delta t)&~=~ \frac{1}{3}f(t)+\frac{2}{3}
\mathtt{FE}(f_{[2]}, \Delta t).
\end{align*}
Here, $\mathtt{FE}(f,\Delta t)$ represents a forward Euler step with size
$\Delta t$.

As an SSP time discretization is a convex combination of forward
Euler, positivity preserving property is granted automatically.

\subsection{Full system}\label{sec:full}
We go back to the full kinetic Cucker-Smale system \eqref{eq:main}.
Using classical splitting method (consult {\it e.g. }\cite{MQ}), we
can separate the system into two components:
the free transport part
\[\partial_tf(t,\x,\v)=-\v\cdot\grad_\x f(t,\x,\v),\]
and the flocking part
\[\partial_tf(t,\x,\v)=-\grad_\v\cdot Q(f,f).\]

The free transport part can be treated using standard methods, for
instance, WENO scheme \cite{Sh}. Note that the choice of method does
not directly affect the accuracy in $\v$. Hence, we omit the details
on this part.

\section{Numerical experiments}\label{sec:example}
In this section, we present some numerical examples to demonstrate the
good performance of the DG scheme applied to kinetic flocking models.

\subsection{Test on rate of convergence}
In this example, we
test the rate of convergence of our DG method on system
\eqref{eq:flocking}. We set
a global influence function $\phi(r)=(1+r)^{-1/2}$, and the following
smooth initial density
\[f_0(x,v)=\begin{cases}
 \exp\left(-\frac{1}{.9-x^2-v^2}\right) & \text{if } x^2+v^2<.9\\
   0& \text{otherwise}.
   \end{cases}\]
As there is no free transport, we set the computational domain
$[-1,1]\times [-1,1]$. Fix the number of partitions on $x$ to be
10. For $v$, we test on $2^{s+2}$ partitions, with
$s=1,\cdots,7$. To satisfy the CFL condition \eqref{eq:sCFL}, we
pick $\Delta t=.1\times2^{-s}$ for second order scheme, and 
$\Delta t=.04\times2^{-s}$ for third order scheme. Denote the
corresponding numerical solution be $f^{[s]}$.

To concentrate on $v$ variable, we integrate $x$ and compare the marginals
\[F^{[s]}(t,v)=\int_{-1}^1f^{[s]}(t,x,v)dx.\]

As the equation has no explicit solutions, we use $F^{[7]}$
as a reference solution. 
The $L^1$ error is computed as
\[e_s(t)=\left\|F^{[s]}(t,\cdot)-F^{[7]}(t,\cdot)\right\|_{L^1_v([-1,1])},\quad s=1,\cdots,6.\]

Table \ref{tab:rate} shows the computational convergence rates
\[r_s=-\log_2 (e_{s+1}/e_s),\quad s=1,\cdots,5\] for $t=0,.5,\cdots,3$. The numerical
results validate the desired order of convergence of the corresponding
schemes. We stop our test at time $t=3$ as the solution is already very
singular in $v$. For larger $t$, $F^{(7)}$ can not be considered as
the reference solution.

\begin{table}[h]
\centering
Second order scheme\\
\vspace{.5em}
\begin{tabular}{c|ccccccccc}
\hline
$t$ & 0 &.5 & 1 & 1.5 &2 & 2.5 & 3\\
\hline

$r_1$ &   1.6837  &  2.0844 &   1.9368  &  1.8460  &  1.2570 &   0.6842 &   0.3966 \\
$r_2$ &   2.2040  &  2.1321 &   2.2030  &  1.9350  &  1.9559 &   1.6517 &   0.9761 \\
$r_3$ &   2.0349  &  2.4708 &   2.3373  &  2.1779  &  1.9197 &   1.8891 &   1.9319 \\
$r_4$ &   1.9877  &  2.2188 &   2.4572  &  2.4522  &  2.2309 &   1.9501 &   1.7383\\
$r_5$ &   2.0554  &  2.0846 &   2.2307  &  2.4309  &  2.5247 &   2.3423 &   2.2672 \\
\hline
\end{tabular}\\
\vspace{1.5em}

Third order scheme\\ 
\vspace{.5em}
\begin{tabular}{c|ccccccccc}
\hline
$t$ & 0 &.5 & 1 & 1.5 &2 & 2.5 & 3\\
\hline
$r_1$ &    4.0841 &   3.6550 &   1.9194 &   1.8906 &   2.9130 &   1.4425 &   0.6367 \\
$r_2$ &    2.4202 &   3.6907 &   3.9546 &   3.3594 &   2.0785 &   2.1196 &   2.5912 \\
$r_3$ &    2.9890 &   2.7490 &   2.9330 &   3.4399 &   3.1831 &   2.9012 &   1.7719 \\
$r_4$ &    2.9954 &   3.0400 &   3.0960 &   3.0208 &   3.1179 &   3.5468 &   2.6046 \\
$r_5$ &    3.0052 &   3.1071 &   3.1116 &   3.0173 &   2.9973 &   3.0637 &   4.2268 \\
\hline
\end{tabular}\\ \vspace{1.5em}
\caption{Computational convergence rates for second and third order DG
schemes at different times.}\label{tab:rate}
\end{table}

\subsection{Capture flocking} 
We consider 1D full kinetic CS model \eqref{eq:main} with
initial density
\[f_0(x,v)=\chi_{|x|<1}\chi_{|v|<.5},\]
where $\chi$ is the indicator function. The influence function is set
to be the same as the previous example: $\phi(r)=(1+r)^{-1/2}$. 
As $\phi$ satisfies \eqref{eq:phi}, the solution should
converge to a flock.

We set the computational domain as follows. In $x$ direction, we
compute $D$ from \eqref{eq:D} and get $D\approx 3.98$. By symmetry,
the support of the solution in $x$ direction lies in $(-2,2)$. We set the computational
domain on $x$ to be $[-2.5,2.5]$ for safety. In $v$ direction, the
variation becomes smaller as time increases. Therefore, $[-.5,.5]$ is an
appropriate domain for $v$. We start the test with mesh size
$40\times40$.

For the time step, the CFL condition \eqref{eq:sCFL} suggests $\Delta
 t<\alpha_1/40$. So, for first and second order schemes, we take
$\Delta t=0.01$. For third order scheme, we take $\Delta t=0.004$.
\begin{figure}[h]
\includegraphics[scale=.7]{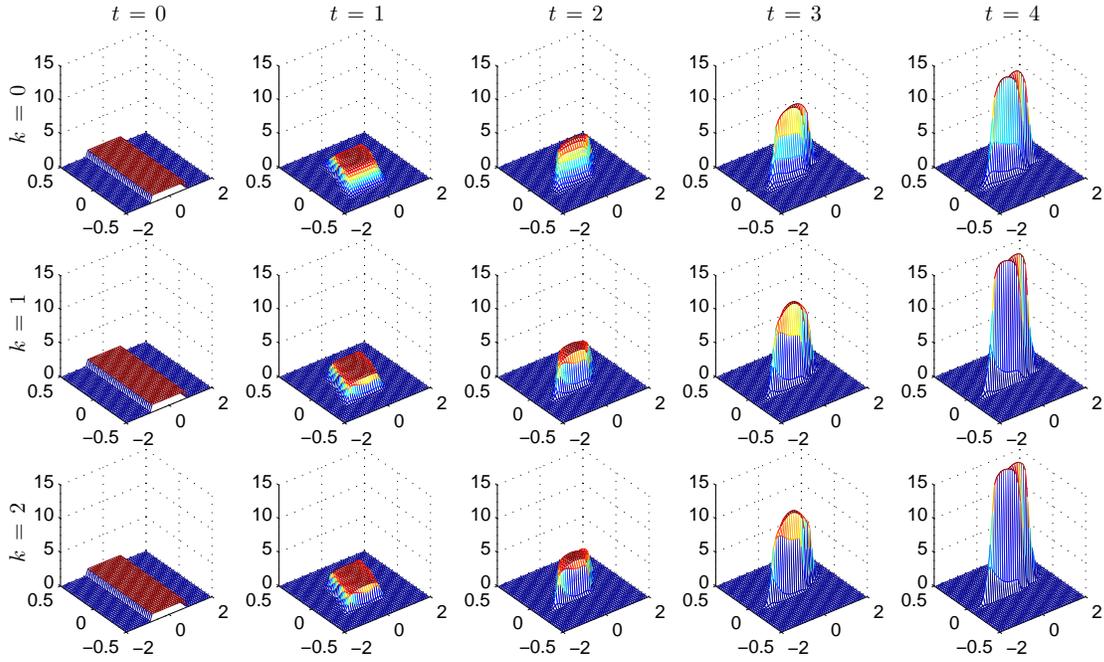}
\caption{Density $f$ at time $t=0,1,2,3,4$ for DG schemes with 
$k=0,1,2$.}\label{fig:Ep2pic1}
\end{figure}

Figure \ref{fig:Ep2pic1} shows the dynamics of density $f$ under
DG schemes using piecewise polynomials of degree $k=0,1,2$. We observe
that all three schemes converge to a flock. On the other hand, high
order schemes concentrate faster than the low order scheme, which is an
indicator of better performance. For a better view, we plot in figure \ref{fig:Ep2pic2} 
the marginal $F(t,v):=\int f(t,x,v)dx$ against $v$ at different
times. We observe that the first order scheme ($k=0$) exhibits a large
numerical diffusion, while higher order schemes are not. There is also
evidence showing third order scheme ($k=2$) is slightly better than
the second order ($k=1$). For instance, at $t=4$, the solution for the
third order scheme is higher around zero, indicating faster concentration.

\begin{figure}[h]
\includegraphics[scale=.7]{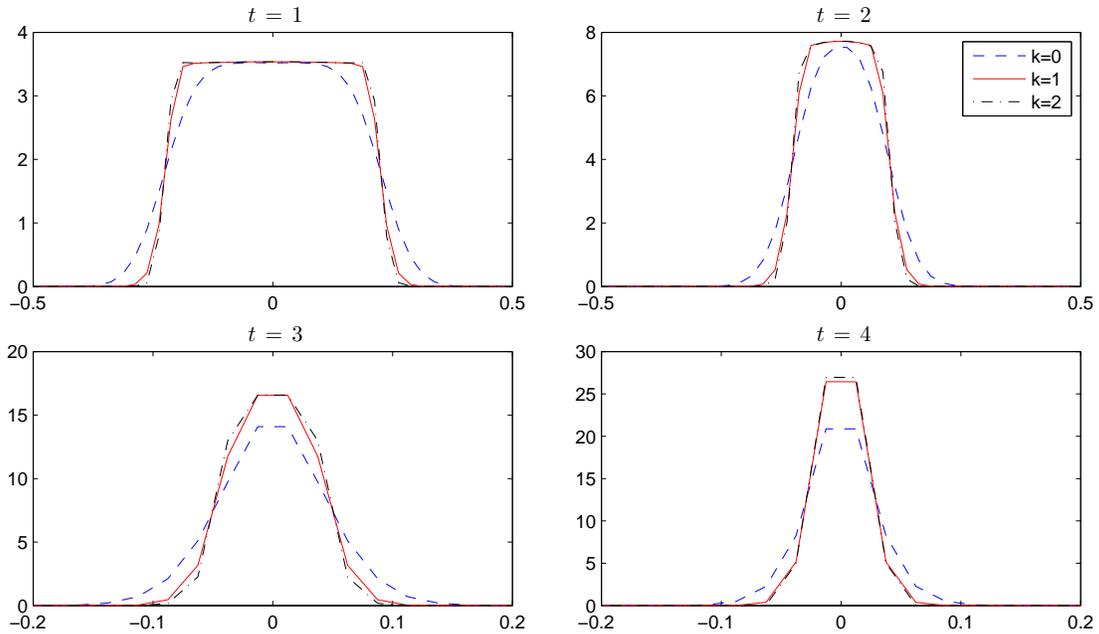}
\caption{$F(t,v)$ at time $t=1,2,3,4$ for DG schemes.}\label{fig:Ep2pic2}
\end{figure}

\subsection{Clusters vesus flocking} 
It is known that flocking is \emph{not} ganranteed if the influence
function is compactly supported, especially when \eqref{eq:phidecay}
does not hold. Multiple clusters might form as time goes. This example
is designed to compare the two asymptotic behaiviors. In fact, our DG
scheme captures both flocking and clusters very well. Let
\[f_0(x,v)=\chi_{-.5<x<-.4}\cdot\chi_{.4<v<.5}+\chi_{.4<x<.5}\cdot\chi_{-.5<v<-.4}.\]
It represents two groups, where the left group is travelling to the
right and the right group is travelling to the left. We consider two
different influence functions:
\[\phi_1(r)=\chi_{r<.8},\quad
\phi_2(r)=\chi_{r<.4}.\]
Both functions are compactly supported. Yet $\phi_1$ is much stronger
than $\phi_2$. In particular, $\phi_1(r)\geq\phi_2(r)$.

Figure \ref{fig:E3P1} shows the evolution of the CS model
under two influence functions. We observe that with strong influence
$\phi_1$, the system converges to a flock. In contrast, with
relatively weak influence $\phi_2$, the interaction is not strong
enough and multiple clusters are forming in large time.

\begin{figure}[h]
\centering
\includegraphics[scale=.7]{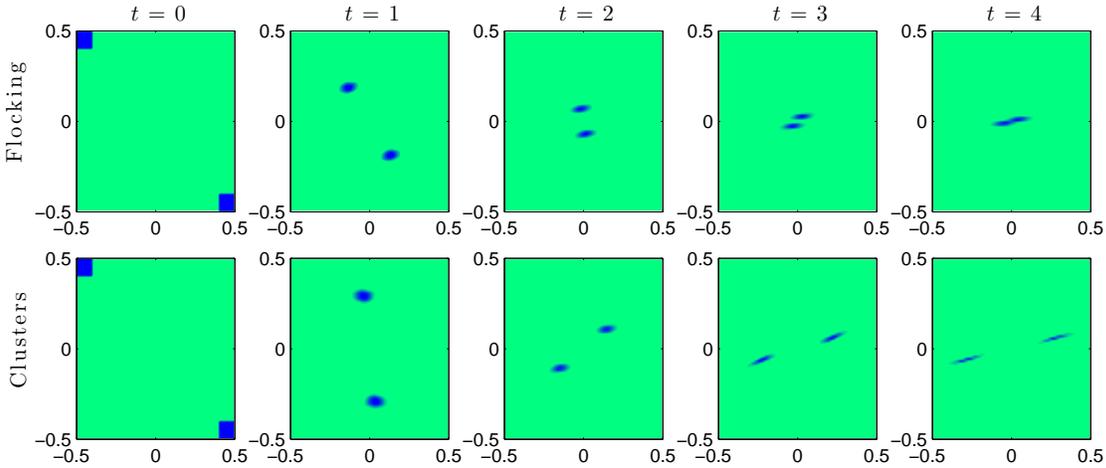}
\caption{Flocking vesus cluster formation.}\label{fig:E3P1}
\end{figure}


\subsection{Cucker-Smale vesus Motsch-Tadmor}\label{sec:vs}
We end this paper with a nice example to compare CS and MT setups
numerically.

Motsch and Tadmor in \cite{MT} discuss a drawback for CS
model which motivates their model. In the particle CS
model, \textit{``the motion of an agent is modified
by the total number of agents even if its dynamics is only influenced
by essentially a few nearby agents.''} For initial configuration far
from equilibrium, CS model has
poor performance in modeling the dynamics. The MT setup
overcomes the drawback by normalizing the influence not by the total
number of agents (or total mass), but by the total influence of each
agent.

The following example is design to compare the results of the two
setups with an initial configuration far from equilibrium. Our DG
schemes have good performances on both setups. It captures the
difference of the two models in kinetic level, which agrees with the
discussion in \cite{MT}.

Consider the initial configuration as a combination of a small group
(with mass .02) and a large
flock (with mass .98) far away
\[f_0(x,v)= \chi_{|x|<.1}\chi_{|v|<.05} + .98\delta(x-5)\delta(v-1),\]
with compactly supported influence function
$\phi(r)=(1-r)^2\chi_{r<1}$. It is easy to see that the large flock never
interact with the small group.

Figure \ref{fig:E4P1} shows numerical results of the evolutions of the
\emph{small group} in both CS and
MT setups. We observe that under CS setup, the
faraway large flock eliminates the interactions inside the small
group. The evolution is almost like a pure transform. In contrast,
MT setup yields the reasonable flocking behavior for the small group. 

\begin{figure}[h]
\centering
\includegraphics[scale=.7]{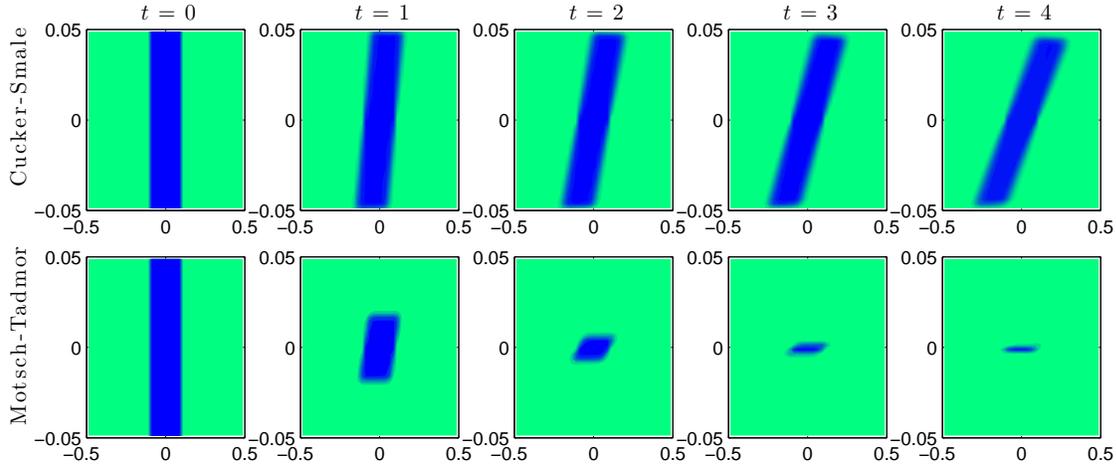}
\caption{Evolution of the small group under 2 models.}\label{fig:E4P1}
\end{figure}


\end{document}